\documentclass{amsart}
\usepackage{amsfonts}
\usepackage{amssymb}

\newtheorem{thm}{Theorem}[section]
\newtheorem{cor}[thm]{Corollary}
\newtheorem{lem}[thm]{Lemma}

\theoremstyle{definition}

\theoremstyle{remark}

\theoremstyle{conjecture}

\numberwithin{equation}{section}
\DeclareMathOperator{\im}{Im}
\DeclareMathOperator{\id}{id}
\DeclareMathOperator{\LT}{\textup{\textsc{lt}}}

\begin{document}
\title[Images of linear derivations and $\mathcal{E}$-derivations]{Images of
linear derivations and linear $\mathcal{E}$-derivations of $%
K[x_{1},x_{2},x_{3}]$}
\thanks{This work is supported by NSFC (No. 11771176)}
\author{Haifeng Tian, Xiankun Du, Hongyu Jia}
\address{School of Mathematics, Jilin University, Changchun 130012, China}%
\email{tianhf1010@foxmail.com, duxk@jlu.edu.cn, jiahy1995@163.com}%


\begin{abstract}
Let $K$ be a field of characteristic zero. We prove that images of a linear $%
K$-derivation and a linear $K$-$\mathcal{E}$-derivation of the ring $%
K[x_{1},x_{2},x_{3}] $ of polynomial in three variables over $K$ are
Mathieu-Zhao subspaces, which affirms the LFED conjecture for linear $K$%
-derivations and linear $K$-$\mathcal{E}$-derivations of $%
K[x_{1},x_{2},x_{3}]$.

\end{abstract}
\maketitle

\section{Introduction}

%

Kernels of derivations of polynomial rings are studied extensively \cite%
{Fre,Now}. Images of polynomial derivations have also attracted researchers'
attention recently due to its relationship with the Jacobian conjecture \cite%
{EWZ}. A key problem is to prove that images of some derivations are
Mathieu-Zhao subspaces (MZ-subspaces for short). The notion of an MZ-subspace
was introduced by Zhao \cite{Zhao12} to study the Jacobian conjecture, which was also called a Mathieu subspace in early
literature.

Let $K$ denote a field of characteristic zero. The image of a $K$-derivation
of the univariate polynomial algebra $K[x]$ is an MZ-subspace of $K[x]$.
This is no longer true for the bivariate polynomial algebra (\cite[Example
2.4]{EWZ}). Van den Essen, Wright and Zhao \cite{EWZ} proved that the image
of every locally finite $K$-derivation of $K[x_{1},x_{2}]$ is an MZ-subspace
of $K[x_{1},x_{2}]$, and the two-dimensional Jacobian
conjecture  is restated that   $\im D$ is an
MZ-subspace of $K[x_{1},x_{2}]$ for every $K$-derivation $D$ with divergence
zero and $1 \in \im D$. They asked that if $\im D$
is an MZ-subspace for each $K$-derivation $D$ with divergence $0$ of $%
K[x_{1},x_{2}]$. The question was answered negatively by Sun \cite{Sun}. It
is natural to wonder images of which $K$-derivations of the polynomial
algebra $K[x_{1},x_{2},\dots,x_n]$ are MZ-subspaces. Zhao \cite{Zhao17}
formulated the LFED Conjecture, which asserts that the images of a locally
finite $K$-derivation and locally finite $K$-$\mathcal{E}$-derivation of a $%
K $-algebra $R$ are an MZ-subspace.

The LFED conjecture is still far from being solved, though it has been
verified for some special algebras \cite{EsZ}.  The LFED conjecture for
polynomial algebras is especially interesting, but it is proved only for
univariate polynomial rings \cite[Theorem 1.3]{Zhao x}. For the bivariate
case, van den Essen, Wright and Zhao \cite[Theorem 3.1]{EWZ} proved the
conjecture for derivations, and the case of $\mathcal{E}$-derivations is
open. When $n=3$, only derivations have been considered. Liu and Sun
confirmed the {LFED} conjecture of $K[x_{1},x_{2},x_{3}]$ for linear locally
nilpotent derivations \cite[Theorem 3.4]{LiuSun},  rank-two locally
nilpotent derivations \cite[Theorem 3.9]{SunLiu} and homogeneous locally
nilpotent derivations \cite[Theorem 5.3]{SunLiu}. For general polynomial
algebras, van den Essen and Sun \cite{EssSun} proved that images of
monomial-preserving derivations of are MZ-subspaces.

In this paper we will prove the LFED conjecture for kinds of special locally
finite derivations and $\mathcal{E}$-derivations, the so-called linear
derivations and linear $\mathcal{E}$-derivations, of $K[x_{1},x_{2},x_{3}]$.
These derivations and $\mathcal{E}$-derivations can be classified into three
types up to conjugation by automorphisms, respectively. Then we check these
types one by one. In Section 3, we prove the image of a linear derivation of
$K[x_{1},x_{2},x_{3}]$ is an MZ-subspace (Theorem \ref{thm D}). In Section
4, we prove that the image of a linear $\mathcal{E}$-derivation of $%
K[x_{1},x_{2},x_3]$ is an MZ-subspace (Theorem \ref{thm delta xyz}). As a
by-product, we get a similar result for the bivariate case (Corollary \ref%
{cor delta x y}).

\section{Preliminaries}

Throughout this paper, let $K$ denote a field of characteristic zero, $%
K^{*}=K\setminus\{0\}$, and $\mathcal{R}$ a commutative $K$-algebra.

A subspace $M$ of $\mathcal{R}$ is called an MZ-subspace of $\mathcal{R}$ if for all $a,b\in \mathcal{R}$,
$a^m\in M$ for all positive integers $m$ implies $ab^m\in M$ for all
sufficiently large integers $m$.

Ideals of $\mathcal{R}$ are MZ-subspaces.
If $M$ is an MZ-subspace of $\mathcal{R}$ and $1\in M$, then $M=\mathcal{R}$.

\begin{lem}
\cite[Proposition 2.7]{Zhao12}\label{lemidealmz} Let $M$ be a subspace of $\mathcal{R}$
containing an ideal $I$ of $\mathcal{R}$. Then $M$ is an MZ-subspace if and only if $%
M/I$ is an MZ-subspace of $\mathcal{R}/I$.
\end{lem}

\begin{lem}
\cite[Lemma 2.5]{EWZ}\label{lemewz} Let $L$ be a field extension of $K$, and $M$
a $K$-subspace of $\mathcal{R}$. If $L\otimes_K M$ is an MZ-subspace of $L\otimes_K \mathcal{R}$%
, then $M$ is an MZ-subspace of $\mathcal{R}$.
\end{lem}

\begin{lem}
\label{lem1eta} Let $\eta: \mathcal{R}\to \mathcal{R}$ be a $K$-linear mapping, and let $L$ be
an extension of $K$. If $\im(\id\otimes \eta)$ is an MZ-subspace of $%
L\otimes_K \mathcal{R}$ then $\im \eta$ is an MZ-subspace of $\mathcal{R}$.
\end{lem}

\begin{proof}
It follows from Lemma \ref{lemewz} and the fact that $\im(\id_K\otimes \eta)= L\otimes_K \im\eta$.
\end{proof}

A $K$-linear mapping $\eta:\mathcal{R}\to \mathcal{R}$ is said to be locally finite if for each $%
a\in \mathcal{R}$ the $K$-subspace spanned by $\eta^{i}(a)$, $i\geq 0$, is finite
dimensional over $K$.

Denote by $K[X]$ the algebra of polynomials in the variables $%
x_{1},x_{2},\ldots ,x_{n}$ over $K$. Let $\eta $ be a $K$-derivation or $K$%
-endomorphism of $K[X]$. If $L$ is a field extension of $K$, then $\id\otimes \eta $
is correspondingly a $L$-derivation or $L$-endomorphism of the $L$-algebra $%
L\otimes _{K}\mathcal{R}$. If each $\eta (x_{i})$ is a linear form of $%
x_{1},x_{2},\ldots ,x_{n}$, then $\eta $ is called linear. If $\eta $ is
linear, then $\id\otimes \eta $ is linear. Linear $K$-derivations and linear $%
K $-endomorphisms of $K[X]$ are locally finite \cite[Example 9.3.1]{Now}.

For $\alpha ,\beta \in K^{n},$ denote by $\alpha \beta $ the usual inner
product of $\alpha$ and $ \beta $. Let $e_{i}$ denote the element of $K^{n}$ with a $1$ in the $i$th
coordinate and $0$'s elsewhere.

Denote by $\mathbb{N}$ the nonnegative integers and $\mathbb{N}^{\ast }$ the
positive integers. For $\beta =(\beta _{1},\beta _{2},\dots ,\beta _{n})\in
\mathbb{N}^{n}$, let $|\beta |=\sum_{i=1}^{n}\beta _{i}$. Write $X^{\beta }$
for $x_{1}^{\beta _{1}}x_{2}^{\beta _{2}}\cdots x_{n}^{\beta _{n}},$ and $%
\alpha ^{\beta }$ for $a_{1}^{\beta _{1}}a_{2}^{\beta _{2}}\cdots
a_{n}^{\beta _{n}}$ for $\alpha =(a_{1},a_{2},\ldots ,a_{n})\in
K^{n}\backslash \{0\}$, where it is understood that $0^{0}=1.$

Fix  a monomial ordering $>$ on $K[X]$ and denote by $\LT(f)$ the leading term of   $f\in K[X]\setminus \{0\}$ (see \cite{Cox}).
\begin{lem}\label{lemeta}
   Let $\eta: K[X]\to K[X]$ be a $K$-linear map  that is graded according to the standard gradation of $K[X]$.  Fix  a monomial ordering $>$ on $K[X]$. Suppose
   $$\LT  (\eta(X^{\beta}))=a_{\beta}X^{\beta}, ~~a_{\beta}\in K^*, ~~~\text{for all}~\beta\in \mathbb{N}^n\setminus \{0\}. $$
   Then $ X^{\beta}\in \im \eta~~\text{for all}~\beta\in \mathbb{N}^n\setminus \{0\}. $
\end{lem}
\begin{proof}
 The proof is by transfinite induction according to $>$. Suppose $X^{\alpha}$ is the minimum nonconstant monomial. First we have $\eta (X^{\alpha})=a_{\alpha}X^{\alpha}+a$, where $a\in K[X]$. Since $\eta $ is
a graded linear map, we have $a=0$ by the minimality of $X^{\alpha}$, whence $X^{\alpha}= a^{-1}_{\alpha}\eta
(X^{\alpha})\in \im\eta $. Given $X^{\beta }\neq 1$, suppose that $X^{\gamma }\in
\im\eta $ if $X^{\beta }>X^{\gamma }\neq 1$. Then we need to prove $X^{\beta
}\in \im\delta $.  Write  $\eta (X^{\beta })= a_{\beta }X^{\beta }+f,$
where $f$ is a homogeneous polynomial and its each term is less than $X^{\beta }
$. By induction hypothesis we have $f\in \im\eta $. Since $a_{\beta }\neq 0$,
we have $X^{\beta }\in \im\eta $, as desired.
\end{proof}

\section{Image of linear $K$-derivations}

Denote by $K[X]$ the algebra of polynomials in the variables $%
x_1,x_2,\ldots,x_n$ over $K$.
Let $D$ be a linear $K$-derivation $K[X]$. Then $D X=XA$ for an $n\times n$
matrix $A$ over $K$, where $X=(x_1,x_2,\ldots,x_n)$ and $DX=(D(x_1),D(x_2),%
\dots,D(x_n))$. For any linear $K$-automorphism $\sigma$ of $K[X]$, $\sigma
D\sigma^{-1}$ is still a linear $K$-derivation of $K[X]$, and the image of $%
D $ is an MZ-subspace of $K[X]$ if and only if the image of $\sigma
\eta\sigma^{-1}$ is an MZ-subspace of $K[X]$. There exists a linear $K$%
-automorphism $\sigma$ of $K[X]$ such that $\sigma D\sigma^{-1}X=XJ$, where $%
J$ is a Jordan matrix similar to $A$ over the algebraic closure of $K$. Thus
we have the following lemma.

\begin{lem}
\label{lem D case123}Suppose $K$ is algebraically closed. Let $D$ be a
linear $K$-derivation of $K[x_{1},x_{2},x_{3}]$. Then $D$ is conjugate by a
linear $K$-automorphism of $K[x_{1},x_{2},x_{3}]$ to one of following
derivations:

\begin{enumerate}
\item $ax_{1}\partial_{x_{1}}+bx_{2}\partial_{x_{2}}+cx_{3}%
\partial_{x_{3}} $;
\item $ax_{1}\partial_{x_{1}}+(bx_{2}+x_{3})\partial_{x_{2}}+bx_{3}%
\partial_{x_{3}}$;
\item $(ax_{1}+x_{2})\partial_{x_{1}}+(ax_{2}+x_{3})%
\partial_{x_{2}}+ax_{3}\partial_{x_{3}}$;
\end{enumerate}
where $a,b,c\in K$.
\end{lem}


%
%
%

\begin{lem}
\label{lem D abcde}\cite[Lemma 3.4]{EWZ} Let $D=\sum_{i=1}^{n}a_{i}x_{i}%
\partial_{x_{i}}$, where $a_{i}\in K~ (1\leq i\leq n)$. Then $\im D$ is an
MZ-subspace of $K[X]$.
\end{lem}

\begin{lem}
\label{lem D abcdee1} Let $D=\sum_{i=1}^{n-2}a_{i}x_{i}%
\partial_{x_{i}}+(a_{n-1}x_{n-1}+x_{n})\partial_{x_{n-1}}+a_{n-1}x_{n}%
\partial_{x_{n}}$, where $a_{i}\in K ~(1\leq i\leq n-1)$. Then $\im D$ is an
MZ-subspace of $K[X]$.
\end{lem}

\begin{proof}
We first prove that $\im D$ contains the ideal $(x_{n})$ generated by $x_n$. Let $$S=\{(\beta_{1},
\beta_{2},\dots,\beta_{n})\in\mathbb{N}^{n}\mid \beta_{n}>0\}.$$
It suffices to prove that $X^{\beta}\in \im D$ for all $\beta\in S$.
Let $\alpha=(a_{1},a_{2},\dots,a_{n-1},a_{n-1})$. Then we have
\begin{equation}  \label{eq D xyz}
D(X^{\beta})=\alpha\beta X^{\beta}+\beta_{n-1}X^{\beta-e_{n-1}+e_n}, \text{\quad
for all\quad}\beta\in\mathbb{N}^{n}.
\end{equation}
For $\beta= (\beta_{1},
\beta_{2},\dots,\beta_{n})\in S$, we need to consider the following cases.

If $\alpha\beta=0 $, then $\alpha(\beta+e_{n-1}-e_n)=0$. It follows  from \eqref{eq D xyz} that
\begin{equation*}
X^\beta=\frac{1}{\beta_{n-1}+1}D(X^{\beta+e_{n-1}-e_n}).
\end{equation*}
Since $\beta\in S$, we have $\beta+e_{n-1}-e_n\in \mathbb{N}^n$, whence $X^{\beta}\in\im D$.

If $\alpha\beta\neq0$, then
  from \eqref{eq D xyz} we have for any $0\leq i\leq \beta_{n-1}-1$,
\begin{equation*}
D(X^{\beta-ie_{n-1}+ie_{n }})=\alpha\beta X^{\beta-ie_{n-1}+ie_{n }}+(\beta_{n-1}-i)X^{\beta-(i+1)e_{n-1}+(i+1)e_{n }},
\end{equation*}
from which it follows that
\begin{equation*}  
X^{\beta-ie_{n-1}+ie_{n }}\equiv \frac{i-\beta_{n-1}}{\alpha\beta}
X^{\beta-(i+1)e_{n-1}+(i+1)e_{n }}~~\pmod{\im D},
\end{equation*}
{for}~$i=0,1,\dots,\beta_{n-1}-1$. Thus
\begin{equation*}
X^{\beta} \equiv (-1)^{\beta_{n-1}}\frac{\beta_{n-1}!}{(\alpha\beta)^{\beta_{n-1}}%
}X^{\beta-\beta_{n-1}e_{n-1}+\beta_{n-1}e_{n }} \pmod {\im D}.
\end{equation*}
By \eqref{eq D xyz}, we get $D(X^{\beta-\beta_{n-1}e_{n-1}+\beta_{n-1}e_{n }})=\alpha\beta
X^{\beta-\beta_{n-1}e_{n-1}+\beta_{n-1}e_{n }} $. Since $\alpha\beta\ne0$, we have $X^{\beta-\beta_{n-1}e_{n-1}+\beta_{n-1}e_{n }}\equiv 0~\pmod{\im D}$. Thus $X^{\beta}\in\im D$. %


We now prove that $\im D$ is an MZ-subspace of $K[X]$. To do this,  by Lemma
\ref{lemidealmz}  we only need to show $\im D/(x_{n})$ is an MZ-subspace of $%
K[X]/(x_{n})$. Note that $K[X]/(x_n)=K[y_1,y_2,\dots,y_{n-1}]$ is a
polynomial ring in the variables $y_1,y_2,\dots,y_{n-1}$. Let $\bar{D}$ be
the derivation induced by $D$ in $K[X]/(x_{n})$. Then $\bar{D}=\sum_{i=1}^{n-1}a_{i}y_{i}\partial_{y_{i}}$ and $\im \bar{D}=\im
D/(x_{n})$. By
Lemma \ref{lem D abcde}, $\im\bar{D}$ is an MZ-subspace of $K[y_1,y_2,\dots,y_{n-1}]$, and so $\im
D/(x_{n})$ is  an MZ-subspace of $K[X]/(x_{n})$,
as desired.
\end{proof}

\begin{lem}
\label{lem D aaaa111} Let $D=\sum_{i=1}^{n-1}(ax_{i}+x_{i+1})%
\partial_{x_{i}}+ax_{n}\partial_{x_{n}}$, where $a\in K^{*}$. Then $\im D$
is the ideal of $K[X]$ generated by $x_{1}, x_{2}, \dots, x_{n}$.
\end{lem}

\begin{proof}

Denote by $(X)$ the ideal generated by $x_{1}, x_{2}, \dots, x_{n}$. It is
obvious that $\im D\subseteq(X)$. To prove the  opposite inclusion it suffices to prove $X^{\beta }\in \im D$ for all $\beta\in\mathbb{N}^n\setminus\{0\}$. It is clear that $D$ is graded linear map. Take the lexicographical order on $K[X]$ with $x_{1}>x_{2}>\cdots>x_{n}$. Then for any $\beta\in\mathbb{N}^n\setminus\{0\}$,
\begin{align*}
\LT(D(X^{\beta })) &= \LT( a|\beta|X^{\beta}+\sum_{1\leq i\leq n-1, \beta_{i}\neq0}\beta_{i}X^{\beta-e_{i}+e_{i+1}})= a{|\beta |}X^{\beta },
\end{align*} and $a|\beta |\ne 0$ since $a\neq 0$. By Lemma \ref{lemeta},
$X^{\beta }\in \im D$ for all $\beta\in\mathbb{N}^n\setminus\{0\}$, as desired.
\end{proof}

\begin{lem}
\label{lem D LN}\cite[Theorem 3.4]{LiuSun} Let $D$ be a linear locally
nilpotent $K$-derivation of $K[x_{1},x_{2},x_{3}]$. Then $\im D$ is an
MZ-subspace of $K[x_{1},x_{2},x_{3}]$.
\end{lem}

\begin{thm}
\label{thm D} Let $D$ be a linear $K$-derivation of $K[x_{1},x_{2},x_{3}]$.
Then $\im D$ is an MZ-subspace of $K[x_{1},x_{2},x_{3}]$.
\end{thm}

\begin{proof}
By Lemma \ref{lem1eta}, without loss of generality, we assume that $K$ is
an algebraically closed field. Then we only need to consider the three
derivations in Lemma \ref{lem D case123}. Lemma \ref{lem D abcde} and Lemma %
\ref{lem D abcdee1} prove the first two cases, respectively. The last case
follows from Lemma \ref{lem D aaaa111} if $a\ne 0$ and Lemma \ref{lem D LN}
if $a=0$.
\end{proof}

\section{image of linear $K$-$\mathcal{E}$-derivation of $%
K[x_{1},x_{2},x_{3}]$}

Denote by $K[X]$ the algebra of polynomials in the variables $%
x_1,x_2,\ldots,x_n$ over $K$.
Let $\phi$ be a $K$-endomorphism of the $K$-algebra $K[X]$. Then $\id- \phi$ is
called a $K$-$\mathcal{E}$-derivation of $K[X]$. Furthermore, $\id- \phi$ is
called linear if $\phi$ is linear.

Let $\phi$ be a linear $K$-endomorphism of $K[X]$. Then $\phi X =XA$ for an $%
n\times n$ matrix $A$ over $K$, where $X=(x_1,x_2,\ldots,x_n)$ and $\phi X
=(\phi (x_1),\phi (x_2),\dots,\phi (x_n))$. Denote by $\delta$   the $K$-$\mathcal{E}$-derivation $\id- \phi$. For any linear $K$-automorphism $%
\sigma$ of $K[X]$, $\sigma \delta\sigma^{-1}=\id-  \sigma \phi\sigma^{-1}$   is still a linear $K$-$\mathcal{E}$-derivation of $K[X]$, and the image of $\delta $ is an MZ-subspace of $K[X]$
if and only if the image of $\sigma \delta \sigma^{-1}$ is an MZ-subspace of $%
K[X]$. There exists a linear $K$-automorphism $\sigma$ of $K[X]$ such that $%
\sigma \phi \sigma^{-1} X =XJ$, where $J$ is a Jordan matrix similar to $A$
over the algebraic closure of $K$. Thus we have the following lemma.

\begin{lem}
\label{lem delta case123} Let $\phi$ be a linear $K$-endomorphism of $%
K[x_{1},x_{2},x_{3}]$. Then $\phi$ is conjugate by a linear $K$-automorphism
of $K[x_{1},x_{2},x_{3}]$ to a $K$-endomorphism of $K[x_{1},x_{2},x_{3}]$
defined by one of the following conditions:

\begin{enumerate}
\item $\phi(x_{1})=ax_{1}, ~\phi(x_{2})=bx_{2}, ~\phi(x_{3})=cx_{3}$;
\item $\phi(x_{1})=ax_{1},~ \phi(x_{2})=bx_{2}+x_{3}, ~\phi(x_{3})=bx_{3}$;
\item $\phi(x_{1})=ax_{1}+x_{2}, ~\phi(x_{2})=ax_{2}+x_{3}, ~
\phi(x_{3})=ax_{3}$;
\end{enumerate}
where $a,b,c\in K$.
\end{lem}

Combining \cite[Lemma 3.2]{EssSun} and \cite[Corollary 3.3]{EssSun} we have
the following lemma.

\begin{lem}
\label{lem delta abcde} Let $\delta= \id -\phi$ be a $K$-$\mathcal{E}$%
-derivation of $K[X]$ and $\phi(x_{i})=a_{i}x_{i}$ with $a_{i}\in K$ for all
$1\leq i\leq n$. Then $\im \delta$ is an MZ-subspace of $K[X]$.
\end{lem}


\begin{lem}
\label{lem delta abcdee1} Let $\delta= \id -\phi$ be a $K$-$\mathcal{E}
$-derivation of $K[X]$, $n\geq2$, such that
\begin{align*}
&\phi(x_{i})=a_{i}x_{i} \text{\quad for all\quad} 1\leq i\leq n-2, \\
&\phi(x_{n-1})=a_{n-1}x_{n-1}+x_{n}, \\
&\phi(x_{n})=a_{n-1}x_{n}
\end{align*}
with $a_{i}\in K$ for all $1\leq i\leq n-1$. Then

\begin{enumerate}
\item $\im\delta $ contains the ideal $(x_{n})$ generated by $x_{n}$;
\item $\im\delta $ is an MZ-subspace of $K[X]$.
\end{enumerate}
\end{lem}

\begin{proof}
(1) Let $\alpha=(a_{1},a_{2},\dots ,a_{n-1},a_{n-1})$. For $\beta =(\beta
_{1},\beta _{2},\dots ,\beta _{n})\in \mathbb{N}^{n}$, we have
\begin{align} \label{eq delta X beta abcde1e}
\delta (X^{\beta })
& =(1-\alpha^{\beta })X^{\beta }-\sum_{i=1}^{\beta _{n-1}}\binom{\beta _{n-1}}{i}%
\alpha^{\beta -ie_{n-1}}X^{\beta -ie_{n-1}+ie_{n}},
\end{align}
where it is understood that the sum is $0$ whenever $\beta _{n-1}=0.$ Let
$$ S=\{ (\beta
_{1},\beta _{2},\dots ,\beta _{n})\in \mathbb{N}^{n}\mid \beta _{n}>0\}.$$ We have to prove that $X^{\beta
}\in \im\delta $ for any $\beta \in S$. We proceed by induction on $\beta $ $%
\in S$ according to the graded lexicographical order of $\mathbb{N}^{n}$ with $e_{1}>e_{2}>
\cdots> e_{n}$. First we have $x_{n}= (1-a_{n-1})^{-1}\delta ( %
x_{n})$ if $a_{n-1}\neq 1$, and $x_{n}= -\delta (x_{n-1})$ if $%
a_{n-1}=1.$ Thus $x_{n}\in \im\delta $. Given  $\beta= (\beta
_{1},\beta _{2},\dots ,\beta _{n}) \in S$, suppose that $%
X^{\gamma }\in \im\delta $ for all $\eta \in S$ with $\beta >\gamma  $. Then
we need to prove $X^{\beta }\in \im\delta $. The proof splits two cases.

If $\alpha^{\beta }\neq 1$, then by (\ref{eq delta X beta abcde1e}), we have
\begin{equation}
X^{\beta }\equiv \frac{1}{1-\alpha^{\beta }}\sum_{i=1}^{\beta _{n-1}}\binom{\beta
_{n-1}}{i}\alpha^{\beta -ie_{n-1}}X^{\beta -ie_{n-1}+ie_{n}}\pmod{\im\delta}.
\end{equation}
When $\beta _{n-1}=0$,   $X^{\beta }\equiv 0\pmod{\im\delta}$.  When $\beta
_{n-1}>0$,  we have $\beta -ie_{n-1}+ie_{n}\in S$, and $ \beta > \beta
-ie_{n-1}+ie_{n} $ for $i=1,2,\ldots ,\beta _{n-1}$. By the
induction hypothesis we have $X^{\beta }\equiv 0\pmod{\im\delta}$.

If $\alpha^{\beta }=1$, then by (\ref{eq delta X beta abcde1e}), we have
\begin{equation*}
\delta (X^{\beta +e_{n-1}-e_{n}})=-\sum_{i=1}^{\beta _{n-1}+1}\binom{\beta
_{n-1}+1}{i}\alpha^{\beta -(i-1)e_{n-1}-e_{n}}X^{\beta -(i-1)e_{n-1}+(i-1)e_{n}}.
\end{equation*}%
When $\beta _{n-1}=0$, we have $X^{\beta }\equiv -\frac{1}{\alpha^{\beta -e_{n}}}%
\delta (X^{\beta +e_{n-1}-e_{n}})\equiv 0\pmod{\im\delta}$  since $\beta\in S$. When $\beta
_{n-1}>0$, we have
\begin{equation*}
X^{\beta }\equiv -\frac{1}{(\beta _{n-1}+1)\alpha^{\beta -e_{n}}}%
\sum_{i=2}^{\beta _{n-1}+1}\binom{\beta _{n-1}+1}{i}\alpha^{\beta
-(i-1)e_{n-1}-e_{n}}X^{\gamma_i}\pmod{\im\delta},
\end{equation*}
where $\gamma_i=\beta -(i-1)e_{n-1}+(i-1)e_{n}$. Since $\gamma_i\in S$ and $\beta > \gamma_i  $ for $i=2,\ldots ,\beta _{n-1}+1$, by
the induction hypothesis we have $X^{\beta }\equiv 0\pmod{\im\delta}$.

(2) We now prove that $\im\delta $ is an MZ-subspace of $K[X]$. To do this,
by Lemma \ref{lemidealmz}, we only need to show $\im\delta /(x_{n})$ is an
MZ-subspace of $K[X]/(x_{n})$. Note that $K[X]/(x_{n})=K[y_{1},y_{2},\dots
,y_{n-1}]$ is a polynomial ring in the variables $y_{1},y_{2},\dots ,y_{n-1}$%
. Let $\bar{\phi}$ be the $K$-endomorphism of $K[X]/(x_{n})$ induced by $%
\phi $. Then $\bar{\phi}(y_{i})=a_{i}y_{i}$ for all $1\leq i\leq n-1$. Let $%
\bar{\delta}=\id-\bar{\phi}$. Then $\im\bar{\delta}=\im\delta /(x_{n})$. By
Lemma \ref{lem delta abcde}, $\im\bar{\delta}$ is an MZ-subspace of $%
K[y_{1},y_{2},\dots ,y_{n-1}]$, and so $\im\delta /(x_{n})$ is an MZ-subspace of $K[X]/(x_{n})$, as desired.
\end{proof}

\begin{cor}
\label{cor delta x y} Let $\delta= \id -\phi$ be a linear $K$-$%
\mathcal{E}$-derivation of $K[x_{1},x_{2}]$. Then $\im \delta$ is an
MZ-subspace of $K[x_{1},x_{2}]$.
\end{cor}

\begin{proof}Without loss of generality we assume that $K$ is algebraically closed and $\phi$ is defined by one of the following conditions: \begin{enumerate}
\item $%
\phi(x_{1})=ax_{1},~\phi(x_{2})=bx_{2}$,
\item $%
\phi(x_{1})=ax_{1}+x_{2},~\phi(x_{2})=ax_{2}$,
 \end{enumerate}
 where $a,b\in K$. Then the two cases  follow  from Lemma \ref{lem delta abcde} and  Lemma \ref{lem delta abcdee1}, respectively.
\end{proof}
%

Let $D$ be a locally nilpotent $K$-derivation of $K[X]$. Then $D$ induces the exponential automorphism $e^D=\sum_{k=0}^{\infty }\frac{D^{k}}{k!}$ of $K[X]$.

From \cite[Corollary 2.4]{Zhao 1-1} we have the following lemma.
\begin{lem}
Let $D$ be a locally nilpotent $K$-derivations of $K[X]$. Then $\im(1-e^{D})=\im D$.
\label{lem1-1}
\end{lem}

In what follows,  write $\mathcal{R}=K[x_1,x_2,x_3]$ and denote by $X$ the triple $(x_{1},x_{2},x_{3})$.

 For $a\in K^*$ we define a $K$-endomorphism  $\phi _{a}$  of $\mathcal{R}$ as follows
\begin{equation}\label{eqphia}
   \phi _{a}(x_{1})=ax_{1},~\phi _{a}(x_{2})=ax_{1}+ax_{2},~\phi _{a}(x_{3})=-%
\frac{1}{2}ax_{1}-ax_{2}+ax_{3},
 \end{equation}%
and denote by $ \delta _{a}$ the  $K$-$
\mathcal{E}$-derivation $\id-\delta_a$.  Then for any $\beta
=(\beta _{1},\beta _{2},\beta _{3})\in \mathbb{N}^{3}$, we have
\begin{equation}
\delta _{a}(X^{\beta })=X^{\beta }-a^{|\beta |}x_{1}^{\beta
_{1}}(x_{1}+x_{2})^{\beta _{2}}(-\frac{1}{2}x_{1}-x_{2}+x_{3})^{\beta _{3}}.
\label{eq delta a xmynzk}
\end{equation}%

Consider the standard gradation  $\mathcal{R}  =\oplus
_{i\in \mathbb{N}}\mathcal{R}_{i}$, where $\mathcal{R}_{i}$ is the set of
homogeneous polynomials of degree $i$. Given a positive integer $m$, let
$\mathcal{B}=\oplus _{i\mid m }\mathcal{R}_{i}$ and $\mathcal{C}%
=\oplus _{i\nmid m}\mathcal{R}_{i}$. Then $\mathcal{R}=%
\mathcal{B}\oplus \mathcal{C}$.


%
%
%
%
%
%

\begin{lem}
\label{lem C} If $a\in K$ is an $ m  $th primitive root of unity, then $\delta _{a}(\mathcal{C})=\mathcal{C}.$
\end{lem}

\begin{proof}Since $\delta_a$ is a graded $K$-linear map of $\mathcal{R}$, we have
 $\delta _{a}(\mathcal{C})\subseteq
\mathcal{C}$. Now we show the opposite inclusion.
It suffices to prove that
each monomial in $\mathcal{C}$ is contained in $\im\delta _{a}$.   If $a=1$, then $m=1$ and $\mathcal{C}=0$. Nothing needs to prove. Suppose $a\ne 1$. We proceed
by induction according to the graded lexicographical order on $K[X]$
with $x_{3}> x_{2}>x_{1}$. First we have $x_{1}=(1-a)^{-1}\delta
_{a}(x_{1})\in \im\delta _{a}$. Given   $X^{\beta }\in \mathcal{C}$, suppose
that $X^{\gamma }\in \im\delta _{a}$ for all $X^{\gamma }\in \mathcal{C}$ with $%
X^{\beta }> X^{ \gamma}$. Then we need to prove $X^{\beta }\in \im\delta $. Note that
\begin{multline*}
 \LT \left( x_{1}^{\beta _{1}}(x_{1}+x_{2})^{\beta _{2}}(-\frac{1}{2}%
x_{1}-x_{2}+x_{3})^{\beta _{3}}\right) \\ =\left( \LT \left( x_{1}\right) \right)
^{\beta _{1}}\left( \LT (x_{1}+x_{2})\right) ^{\beta _{2}}\left( \LT (-\frac{1}{2%
}x_{1}-x_{2}+x_{3})\right) ^{\beta _{3}}=X^{\beta } .
\end{multline*}
Thus by (\ref{eq delta a xmynzk}) we have
$
\delta _{a}(X^{\beta })=(1-a^{|\beta |})X^{\beta }+f,
$
where $f\in \mathcal{C}$   and its each term   is contained in $\mathcal{C}$ and is
less than $X^{\beta }$. By induction hypothesis $f\in \im\delta _{a}$. Thus $%
X^{\beta }\in \im\delta _{a}$, as desired.
\end{proof}

\begin{lem}
\label{lem DB}  If $a\in K$ is an $ m  $th primitive root of unity, then $\delta _{a}(\mathcal{B})=D(\mathcal{B})$, where $%
D=x_{1}\partial _{x_{2}}-x_{2}\partial _{x_{3}}$ is a $K$-derivation of $%
\mathcal{R}$.
\end{lem}

\begin{proof}
One can check   $e^{D}=\phi _{1}$, where $\phi _{1}$ is defined in (\ref%
{eqphia}) with $a=1$.  Since $\delta _{1}=\id -\phi _{1}$ is a
graded $K$-linear map of $\mathcal{R},$ we have $\delta _{1}(%
\mathcal{B})\subseteq \mathcal{B}$ and $\delta _{1}(\mathcal{C})\subseteq
\mathcal{C}$, whence $\im\delta _{1}=\delta _{1}(\mathcal{B})\oplus \delta
_{1}(\mathcal{C})$. It follows that $\delta _{1}(\mathcal{B})=\mathcal{B}%
\cap \im\delta _{1}$. Similarly, $D(\mathcal{B})=\mathcal{B}\cap \im D$. By
Lemma \ref{lem1-1}, we have $\delta _{1}(\mathcal{B})=D(\mathcal{B})$.
From (\ref{eq delta a xmynzk}), we can see that $\delta _{a}(X^{\beta })=\delta
_{1}(X^{\beta })$ for all $X^{\beta }\in \mathcal{B}$, which implies that $%
\delta _{a}(\mathcal{B})=\delta _{1}(\mathcal{B}).$ Thus $\delta _{a}(%
\mathcal{B})=D(\mathcal{B})$.
\end{proof}

\begin{lem}
\label{lem delta>D}  If $a\in K$ is an $ m  $th primitive root of unity, then $\im\delta _{a}\supset \im D$.
\end{lem}

\begin{proof}
Since $D(\mathcal{C})\subseteq \mathcal{C}$, we have by Lemma \ref{lem C} and Lemma
\ref{lem DB},
\begin{equation*}
\im\delta _{a}=\mathcal{C}\oplus D(\mathcal{B})\supset D(\mathcal{C})\oplus
D(\mathcal{B})=\im D,
\end{equation*}%
as desired.
\end{proof}

Let $\omega =(-1,0,1)$. For a nonzero polynomial $f=\sum a_{\beta
}X^{\beta }\in \mathcal{R}$ we define $\deg _{\omega }f=\max \{\omega \beta \mid $ $%
a_{\beta }\neq 0\}$ and set $\deg _{\omega }0=-\infty .$ Then $\deg
_{\omega }$ is a degree function on $\mathcal{R}$.

\begin{lem}
\label{lem m>k D}\cite[Lemma 3.2]{LiuSun} Let $D=x_{1}\partial
_{x_{2}}-x_{2}\partial _{x_{3}}$ be a $K$-derivation of $\mathcal{R}$.  If $\deg _{\omega }X^{\beta }<0$, then $X^{\beta }\in \im D$.
\end{lem}

From   the proof of \cite[Lemma 3.4]{LiuSun}, we can formulate
the following fact.

\begin{lem}
\label{lem f^i imD} Let $D=x_{1}\partial
_{x_{2}}-x_{2}\partial _{x_{3}}$ be a $K$-derivation of $\mathcal{R}$.  If $%
f^{i}\in\im D$ for all $i\geq1$, then $\deg_{\omega}f<0$.
\end{lem}

\begin{lem}
\label{lem m>k delta}Let $a\in K$ be an $ m  $th primitive root of unity.  If $\deg_{\omega}X^{\beta}<0$, then $X^{\beta}\in \im  \delta_{a}$.
\end{lem}

\begin{proof}
From Lemma \ref{lem m>k D}, $X^{\beta}\in  \im D$. Then $X^{\beta}\in \im\delta_{a}$ by Lemma \ref{lem delta>D}.
\end{proof}

\begin{lem}
\label{lem f^i im delta}  Let $a\in K$ be an $ m  $th primitive root of unity.
If $f^{i}\in\im\delta_{a}$ for all $i\geq1$, then $\deg_{\omega}f<0$.
\end{lem}

\begin{proof}
Suppose the lemma were false. Then we could find an $f\in
\mathcal{R}$ which satisfies $\deg _{\omega }f\geq 0$ and $f^{i}\in
\im\delta _{a}$ for all $i\geq 1$. Since $a^{ m  }=1$, we have $\delta
_{a}(x_{1}^{i m  })=x_{1}^{i m  }-a^{i m  }x_{1}^{i m  }=0$
  for all  $i\geq 1$, whence
\begin{equation}
x_{1}^{i m  }\in \ker \delta _{a}\text{\quad for all\quad }i\geq 1.
\label{eq ker delta a}
\end{equation}%
Let $l=\deg _{\omega }f$. Since $\ker \delta _{a}$ is a subalgebra of $%
\mathcal{R}$ and $\im\delta _{a}$ is a submodule
of $\mathcal{R}$ over $\ker \delta _{a}$, we have
\begin{equation*}
(x_{1}^{l m  }f^{ m  })^{i}=x_{1}^{il m  }f^{i m  }\in \im%
\delta _{a}\text{\quad for all\quad }i\geq 1.
\end{equation*}%
Since $\deg _{\omega }(x_{1}^{l m  }f^{ m  })=-l m  +l m  =0$%
, replacing $f$ by $x_{1}^{l m  }f^{ m  }$ we can assume that $\deg
_{\omega }f=0$. Write
\begin{equation*}
f=f_{0}+f_{-1}+f_{-2}+\cdots ,
\end{equation*}%
where $f_{j}$ is the $\omega $-homogeneous component of degree $j$ of $f$.
Let $g_{i}=f^{i}-f_{0}^{i}$. Then $\deg _{\omega }g_{i}<0$. By Lemma \ref{lem m>k delta}, $g_{i}\in \im\delta _{a}$. Thus
\begin{equation*}
f_{0}^{i}=f^{i}-g_{i}\in \im\delta _{a}\text{\quad for all\quad }i\geq 1.
\end{equation*}%
Then replacing $f$ by $f_{0}$, we may assume that $f$ is $\omega $-homogeneous with $\deg _{\omega }f=0$.
Denote by $\bar{f}$ the highest homogeneous component of $f$ with respect to the
ordinary degree. Since $\delta _{a}(\mathcal{R}_{i})\subseteq \mathcal{R}_{i}
$ and $f^{i}\in \im\delta _{a}$ for any $i\geq 1$, we can get $\bar{f}%
^{i}\in \im\delta _{a}$. Replacing $f$ by $\bar{f}$, we may assume that $f$
is homogeneous with respect to the ordinary degree. Replacing $f$ by $%
f^{ m  }$ if necessary, we may also assume that $f^{i}\in \mathcal{B}$
for any $i\geq 1$. Since $\im\delta _{a}\cap \mathcal{B}=\delta _{a}(\mathcal{%
B})$, we can get that $f^{i}\in \delta _{a}(\mathcal{B})$. By Lemma \ref{lem
DB}, $f^{i}\in D(\mathcal{B})$ for any $i\geq 1$. By Lemma \ref{lem f^i imD}%
, $\deg _{\omega }f<0$, which contradicts $\deg _{\omega }f=0$.
\end{proof}

\begin{lem}
\label{lem delta last case} If $a\in K$ be a $m$th primitive root of unity,  then $\im\delta _{a}$ is an MZ-subspace of $%
\mathcal{R}$.
\end{lem}

\begin{proof}
Suppose $f\in \mathcal{R}$ is such that $f^{i}\in \im\delta _{a}$
for any $i\geq 1$. Then $\deg _{\omega }f<0$ by Lemma \ref{lem f^i im delta}.
For any $g\in \mathcal{R}$, we have
\begin{equation*}
\deg _{\omega }gf^{i}=\deg _{\omega }g+i\deg _{\omega }f<0\quad \text{for
all sufficiently large integers }i.
\end{equation*}
It follows from Lemma \ref{lem m>k delta} that $gf^{i}\in \im\delta _{a}$
for all sufficiently large integers $i$. Thus $\im\delta _{a}$ is an
MZ-subspace of $\mathcal{R}$.
\end{proof}

\begin{thm}
\label{thm delta xyz} Let $\delta= \id -\phi$ be a linear $K$-$%
\mathcal{E}$-derivation of $ K[x_1,x_2,x_3] $. Then $\im \delta$ is an
MZ-subspace of $K[x_1,x_2,x_3] $.
\end{thm}

\begin{proof}
By Lemma \ref{lem1eta}, we assume, without loss of generality, that $K$ is
an algebraically closed field. Then we only need to consider the three
cases in Lemma \ref{lem delta case123}. Lemma \ref{lem delta abcde}
and Lemma \ref{lem delta abcdee1} prove the first two cases. We now prove
the last case.

Suppose that $a$ is an $m$th primitive root of unity.  Consider the following matrices
\[J=
\begin{pmatrix}
a & 0 & 0 \\
1 & a & 0 \\
0 & 1 & a
\end{pmatrix},\quad \text{and}\quad
 A=
\begin{pmatrix}
a & a & -\frac{1}{2}a \\
0 & a & -a \\
0 & 0 & a%
\end{pmatrix} .
\] Then we see that
 \[
\phi (x_{1},x_{2},x_{3})=(x_{1},x_{2},x_{3})J, \quad \text{and}\quad
\phi_a(x_{1},x_{2},x_{3})=(x_{1},x_{2},x_{3})A.
 \]
Since $A$ is similar to $J$,
 $\delta $ is conjugate to $\delta _{a}$ by a linear $K$-automorphism of $K[x_{1},x_{2},x_{3}]$.  By Lemma \ref{lem delta last case}, $\im\delta$ is an MZ-subspace.

Suppose $a$ is not a root of unity. Take the lexicographical order on $K[X]$ with $x_1 > x_2> x_3$. When $a\ne0$,
\begin{align*}
 \LT (\phi (X^{\beta }))&=\LT \left( (ax_{1}+x_{2})^{\beta
_{1}}(ax_{2}+x_{3})^{\beta _{2}}(ax_{3})^{\beta _{3}}\right) \\
&= \left( \LT \left( ax_{1}+x_{2}\right) \right) ^{\beta _{1}}\left(
\LT (ax_{2}+x_{3})\right) ^{\beta _{2}}\left( \LT (ax_{3})\right) ^{\beta
_{3}}\\&=a^{|\beta |}X^{\beta }.
\end{align*} Thus
$\LT (\delta (X^{\beta }))=(1-a^{|\beta|})X^{\beta}$  and  $1-a^{|\beta|}\ne0$
   for all $\beta\in \mathbb{N}^n\setminus\{0\}$, since $a$ is not a root of unity. By Lemma \ref{lemeta}, $\im \delta$ is an MZ-subspace.
%
\end{proof}



\end{document}